\documentclass[11pt]{amsart}
\usepackage{amsmath, amssymb, amsthm,color}
\usepackage{amssymb,longtable}
\usepackage{amsfonts}
\usepackage{amsmath} 
\usepackage{hyperref} 
\usepackage{latexsym}
\usepackage{array}
\usepackage{amssymb}
\theoremstyle{plain}

\numberwithin{equation}{section}
\newtheorem{thm}{Theorem}[section]
\newtheorem{prop}[thm]{Proposition}
\newtheorem{lemm}[thm]{Lemma}

\newtheorem{defi}[thm]{Definition}
\newtheorem{claim}[thm]{Claim}

\theoremstyle{remark}
\newtheorem{rema}[thm]{Remark}

\renewcommand{\div}{{\rm div}}
\newcommand{\curl}{{\rm curl}} 
 
\newcommand{\Z}{\mathbb{Z}}  
  
\newcommand{\N}{\mathbb{N}}

\newcommand{\R}{\mathbb{R}}

\newcommand{\Supp}{\rm Supp }

\newcommand{\D}{\Delta}
\newcommand{\bq}{\begin{equation}}
\newcommand{\eq}{\end{equation}}

\begin{document}

\title[On the wellposedness of the Navier-Stokes-Maxwell system]{Wellposedness of the Navier-Stokes-Maxwell equations}

\author{Pierre Germain}
\address{The Courant Institute,\\ New-York University\\
Mercer Street\\   USA}
\email{pgermain@cims.nyu.edu}
\author{Slim Ibrahim}
\address{Department of Mathematics and Statistics,\\ University of Victoria\\
 PO Box 3060 STN CSC\\   Victoria, BC, V8P 5C3\\ Canada}
\email{ibrahim@math.uvic.ca}
\urladdr{http://www.math.uvic.ca/~ibrahim/}
\author{Nader Masmoudi}
\address{The Courant Institute,\\ New-York University\\
Mercer Street\\   USA}
\email{masmoudi@cims.nyu.edu}
\thanks{P. Germain is partially supported by NSF grant \# 1101269 and a start-up grant from the Courant Institute.}
\thanks{S. Ibrahim is partially supported by NSERC grant \# 371637-2009 grant}
\thanks{N. Masmoudi is partially supported by NSF grant \# 1211806 grant.}

\begin{abstract}
We study the local and global wellposedness of a full system of Magneto-Hydro-Dynamic 
equations. The system is a coupling of the forced (Lorentz force) 
incompressible Navier-Stokes equations with the Maxwell equations 
through Ohm's law for the current. 
We show the local existence of mild solutions for arbitrarily 
large data in a space similar to the scale invariant spaces classically used for Navier-Stokes. 
These solutions are global if the initial data are small enough. 
Our results not only simplify and unify the proofs for the space dimensions
two and three but also refine those in \cite{IK}. The main simplification comes from an 
{\it a priori} $L^2_t(L^\infty_x)$ estimate for solutions of the forced Navier-Stokes equations.
\end{abstract}
\maketitle
\section{Introduction}

The purpose of this paper is the study of  the following full Magneto-Hydro-Dynamics system (MHD):
\begin{equation}
\label{NS}
\ \left\{
\begin{array}{rclll}
\frac{\partial v}{\partial t} +   v\cdot\nabla  v - \nu
    \Delta v +\nabla p & =&     j {{\times}}    B 
 \\
\partial_t E  - \curl\, B&=&-j \\
\partial_tB  +  \curl\, E
  & = & 0 &
\\
 \div v=\div B &=&0
\\
 \sigma( E + v {{\times}}    B) &=&j
\end{array}
\right.
\end{equation}
subject to the initial data 
\begin{equation*}
 \label{ini}
v_{|t=0} = v^0, \quad B_{|t=0} = B^0, \quad E_{|t=0}= E^0.
\end{equation*}
Here, $v,E,B:\R^+_t\times\R^d_x\longrightarrow \R^3$ are vector fields defined on $\R^d$  ($d=2$ or $3$).
The vector field  $v=(v_1,...,v_d)$ is the velocity of the fluid, $\nu$ its viscosity and  the scalar function  
$p$  stands for  the   pressure. The vector fields $E$ and $B$ are the electric and magnetic fields, respectively, and 
$j$ is the electric current  given by  Ohm's law (the fifth equation of the system, where $\sigma$ is the electric 
resistivity). 
The force term $j \times B $ in the Navier-Stokes equations comes from the Lorentz force under a quasi-neutrality assumption of the net charge carried by the fluid. 
Note that the pressure $p$ can be  recovered from $v$ and $j\times B$ via an explicit 
Cald\'eron-Zygmund operator (see \cite{Chemin95} for instance).
The second  equation in \eqref{NS} is the Amp\`ere-Maxwell equation for an electric field $E$. The third equation is nothing but Faraday's law.  
For a detailed introduction to MHD, we refer to Davidson \cite{Davidson01} and Biskamp \cite{Biskamp93}. 

Note that in the 2D case, the functions $v$, $E$, $B$, and $j$ are defined on the whole space $\R^2$ with values in $\R^3$. In this case, the operator $\nabla$  is given by
$$
\nabla=(\partial_{x_1},\partial_{x_2},0)^T.
$$
Thus  
$$
\operatorname{div} v:=\partial_{x_1}v_1+\partial_{x_2}v_2,\quad \nabla p:=(\partial_{x_1}p,\partial_{x_2}p,0)^T,
$$
and
$$
 \hbox{curl}\,F:=(\partial_{x_2}F_3,-\partial_{x_1}F_3,\partial_{x_1}F_2-\partial_{x_2}F_1)^T.
$$

\bigskip

{\it In the following, we take $\sigma = \nu = 1$ to alleviate the notations.}

Multiply the Navier-Stokes equations in  \eqref{NS} by $v$, the Amp\`ere-Maxwell equations by $(B,E)^T$ and integrate (using the divergence free condition on the velocity); this gives the formal energy identity 
\begin{eqnarray*} 
\label{energy}
\frac12  \frac{d}{dt} \big[\|v\|_{L^2}^2 +\|B\|_{L^2}^2 +\|E\|_{L^2}^2
   \big]   + \|j\|_{L^2}^2  + \|\nabla v \|_{L^2}^2 = 0. 
\end{eqnarray*}
This identity shows that the energy is dissipated by the viscosity and the electric 
resistivity.  It also suggests that one should be able to construct a global finite energy weak solution (\`a la Leray) for data lying in  $L^2(\R^d)$. 
However, this intuitive expectation remains an interesting open problem for \eqref{NS} in both dimensions $d=2,3$.  
Indeed, given a standard approximating scheme, it is hard to obtain compactness of the solutions, especially for the magnetic field due to the hyperbolicity of  Maxwell's equations. 
In dimension 2, the equation is energy critical, but running a fixed point argument for data $(v^0,E^0,B^0)$ only in $L^2(\R^d)^3$ seems very difficult due to the 
term $E\times B$.\\ 

Imposing more regularity on the initial electro-magnetic field, existence results are known. Recently, for initial data 
$(v^0,E^0,B^0)\in L^2(\R^2)\times \big(H^s(\R^2)\big)^2$ with $s>0$, Masmoudi in \cite{Masmoudi10jmpa} proved the existence and uniqueness of global strong solutions to \eqref{NS}.  
His proof relies on the use of the energy inequality combined with a logarithmic inequality  that enabled him to upper estimate the $L^\infty$ norm of the velocity field by 
the energy norm and higher Sobolev norms. It is also interesting to note 
that the proof in \cite{Masmoudi10jmpa} does not use the divergence free condition of the magnetic field, nor the decay property of the linear part coming from Maxwell's equations, 
namely
\begin{equation}
\label{Maxwell}
 \ \left\{
\begin{array}{rclll}
\frac{\partial E}{\partial t}    - \hbox{curl}\, \,B+ E
  & = &f, \\
\\
\frac{\partial B}{\partial t}   +  \hbox{curl}\, \,E
  & = & 0,\\
\\
\nabla\cdot B   &=&0.  
\end{array}
\right.
\end{equation}

Another line of research was pursued by Ibrahim and Keraani \cite{IK} who considered data $(v_0,E_0,B_0) \in \dot{B}^{1/2}_{2,1} (\mathbb{R}^3) \times \big( \dot{H}^{1/2}(\mathbb{R}^3)\big)^2$ in dimension $d=3$, and $(v_0,E_0,B_0) \in \dot{B}^{0}_{2,1} (\mathbb{R}^2) \times \big(L^2_{log}(\mathbb{R}^2)\big)^2$ in dimension $d=2$ (see below for the definiton of these functional spaces).
These authors built up strong solutions by using parabolic regularization arguments giving control of the $L^\infty$ norm of the velocity field of the solution. More recently, Ibrahim and Yoneda constructed local in time solution for non-decaying initial data on the torus. See \cite{IY} for more details. 

In this paper, we follow up on the work of Ibrahim and Keraani by running a fixed point argument to obtain mild solutions, but taking the initial velocity field in the natural Navier-Stokes space $H^{\frac d2-1}$. Our main theorem extends the earlier results that were mentioned in many respects: the regularity of the initial velocity and electromagnetic fields is lowered, and we unify the proofs in the cases of space dimension 2 and 3. One of the key ingredients will be to use an $L^2 L^\infty$ estimate on the velocity field, which simplifies greatly the fixed point argument; in particular, the weak decay for the electromagnetic field is not needed any more in dimension 3.

Before stating our main result, we need a few definitions.

{ \begin{defi}
First, let $\mathcal{P}$ denote the Leray projection on divergence-free vector fields.

A function $\Gamma:=(v,E,B)$ with ${\rm div}(v)={\rm div}(B)=0$ is said to be a mild solution on a time interval $[0,T]$ of the full MHD problem \eqref{NS} 
if $\Gamma\in {\mathcal C}([0,T],\dot H^{\frac d2-1})$ and satisfies the integral equation
$$
\Gamma(t)=e^{t{\mathcal A}}\Gamma(0)+\int_0^te^{(t-t'){\mathcal A}}{\mathcal N}(\Gamma(t'))\;dt',
$$
with 
$$
{\mathcal A}=\left(
\begin{array}{ccc}
\Delta&0&0\\
0&-I&\curl\\0&-\curl &0
\end{array}
\right)
$$
and ${\mathcal N}(\Gamma)=\big({\mathcal P}(-\nabla(v\otimes v)+E\times B+(v\times B)\times B), -v\times B,0\big)^T$. 
\end{defi}
The functional analytic framework we will use is the following.
\begin{defi}
Let $\Delta_q$ denote the dyadic 
frequency localization operator defined in section 2.
For $s, t\in\R$ and $\alpha\geq0$ define the space ${\dot H}^{s,t}_{\alpha}$ by its norm
$$
 \|\phi\|_{\dot H^{s,t}_\alpha   }^2:= 
\sum_{q\leq 0}2^{2qs}\|\Delta_q\phi\|_{L^2}^2+ 
\sum_{q>0}q^\alpha2^{2qt}\|\Delta_q\phi\|_{L^2}^2.
 $$
We will also use the short-hands
$$
\dot{H}^s = \dot{H}^{s,s}_0 \quad,\quad \dot H^s_{\rm log}:=\dot H^{s,s}_1 \quad \mbox{and} \quad \dot H^{s,t} := \dot H^{s,t}_0.
$$ 
Finally, define $\tilde L^r_T{   {\dot H}^{s,t}_\alpha}$ by its norm
 $$
 \|\phi\|_{  \tilde L^r_T{   {\dot H}^{s,t}_\alpha}   }^2:= 
\sum_{q\leq 0}2^{2qs}\|\Delta_q\phi\|_{L^r_TL^2}^2+ 
\sum_{q>0}q^\alpha2^{2qt}\|\Delta_q\phi\|_{L^r_TL^2}^2,
 $$
with obvious generalizations to $\tilde L^r_T  {\dot H}^{s}$ etc...
\end{defi}

The space $ \dot H^s_{\rm log}$ is articulated on the standard 
homogeneous Sobolev space $\dot H^s$ with an extra logarithmic 
weight for the high frequency part. 
The space $\dot H^{s,t}$ is nothing but the usual 
Sobolev space $\dot H^t$ for high frequencies while it behaves like 
$\dot H^{s}$ for low frequencies. If $s>t$, it is not difficult to 
see that $\dot H^{s,t}=\dot H^{s}+\dot H^{t}$.
The $\tilde L$ spaces were first used by 
Chemin and Lerner \cite{CL95}.

Our main result can  be stated as follows.

\begin{thm}
\label{MAIN} 
\begin{itemize}
 \item
In dimension two and for any 
$$
\Gamma^0:=(v^0, E^0,B^0) \in L^2(\mathbb R^2)\times L^2_{\rm log}(\mathbb R^2)
\times L^2_{\rm log}(\mathbb R^2),
$$
there exists $T>0$ and a unique mild solution $\Gamma=(v,E,B)$ of 
\eqref{NS} with initial data $\Gamma^0$ and  
\begin{equation*}
\begin{split}
& v \in \tilde L^\infty(0,T; L^2) \cap L^2(0,T; \dot H^1\cap L^\infty) \\
& E\in \tilde L^\infty(0,T;  L^2_{\rm   log})\cap L^2(0,T;L^2_{\rm   log}) \\
& B \in {\tilde L}^\infty(0,T; L^2_{\rm   log}) \cap  L^2(0,T;\dot{H}^{1,0}).
\end{split}
\end{equation*}
Moreover, the solution is global (i.e. $T=\infty$) if the initial data is sufficiently small in $L^2 \times L^2_{\rm log} \times L^2_{\rm log}$.
\item 
In dimension three and for any 
$$
\Gamma^0:=(v^0, E^0,B^0) \in \dot H^\frac12(\mathbb R^3)\times \dot H^\frac12(\mathbb R^3)\times \dot H^\frac12(\mathbb R^3)
$$ 
there exists $T>0$ and a unique mild solution $\Gamma=(v,E,B)$ of \eqref{NS} with initial data $\Gamma^0$ and  
\begin{equation*}
\begin{split}
& v \in \tilde L^\infty(0,T; \dot H^\frac12) \cap L^2(0,T; \dot H^\frac32\cap L^\infty) \\
& E\in \tilde L^\infty(0,T;    \dot H^\frac12)\cap L^2(0,T;\dot H^\frac12) \\
& B \in {\tilde L}^\infty(0,T; \dot H^\frac12) \cap L^2(0,T;    \dot{H}^{\frac32,\frac12}).
\end{split}
\end{equation*}
Moreover, the solution is global (i.e. $T=\infty$) if the initial data is 
sufficiently small in $\dot H^\frac12\times \dot H^\frac12\times \dot H^\frac12$.

\end{itemize}
\end{thm}


In dimension two, the extra logarithmic regularity is needed to 
estimate the term $E\times B$ appearing in the Navier-Stokes equations.

In dimension three, the control of $B$ in $L^2(0,T;    \dot{H}^{\frac32,\frac12})$ is not needed to close the fixed point estimate, but we added it for completeness.

The rest of this paper is organized as follows. In the next section we define some further tools needed in the proof.
In Section 3, we detail the linear (parabolic regularity) and nonlinear (product law) estimates needed in the proof of the main theorem. 
The main theorem is then proved in Section 4. Finally, the proofs of some technical estimates are given in the appendix.
\begin{center}
{Acknowledgement}
\end{center}
S. Ibrahim thanks all the members of the ``UFR de Math\'ematiques" at \'Ecole Normale Sup\'erieure de Paris for their great hospitality to accomplish this work during his visit there. 

\section{Notations and functional spaces} Throughout this work we use the following notation.
\begin{enumerate}
\item For any positive  $A$ and $B$  the notation  $A \lesssim  B$ means that there exists a positive constant $C$ such that $A\le CB$. 
\item  $c$ will always denote an absolute constant $0<c<1$.
\item For any tempered distribution $u$,  both $ \hat u$  and $\mathcal F u$ denote the Fourier transform of $u$.
\item For every $p\in [1,\infty]$, $\|\cdot\|_{L^p}$ denotes the norm in the Lebesgue space $L^p$.
\item For any normed space ${\mathcal X}$, the mixed space-time Lebesgue space $L^p([0,T],\mathcal X)$ denotes the space 
of functions $f$  such that for almost all $t \in (0,T)$, 
$f(t)\in\mathcal X$ and   $ \|f(t)  \|_{\mathcal X}  \in L^p(0,T)  $. The notation $L^p([0,T],\mathcal X)$ is often shortened to $L^p_T\mathcal X$.
\end{enumerate}

Let us recall the well-known Littlewood-Paley decomposition and the corresponding  cut-off operators. 
There exists a radial positive  function   $\varphi\in\mathcal{D}(\R^d\backslash{\{0\}})$ such that
\begin{eqnarray*}
\label{lpfond1}
\sum_{q\in\mathbb Z} \varphi (2^{-q}\xi)& =& 1\qquad \forall\, \xi\in\mathbb R^d\setminus\{0\},
\\
\Supp\ \varphi(2^{-q}\cdot)\cap \Supp\ \varphi(2^{-j}\cdot)&=&\emptyset, \qquad \forall\, |q-j|\geq 2.
\end{eqnarray*}
For every  $q\in\mathbb Z$ and $v\in{\mathcal S}'(\R^d)$ we set 
$$
\Delta_qv=\mathcal{F}^{-1} \varphi(2^{-q}\xi) \hat v (\xi) \quad\hbox{ and  }\;
S_q=\sum_{ j=-\infty}^{ q-1}\Delta_{j}.
$$
Bony's decomposition \cite{Bony81} consists in splitting the product  $uv$ into three parts\footnote{ It should be said that this decomposition is true in the class of distributions for which 
$\sum_{q\in\mathbb Z}\Delta_q=I$.
For example, polynomial functions do not belong to this class.}: 
$$
uv=T_u v+T_v u+R(u,v),
$$
with
$$
T_u v=\sum_{q}S_{q-1}u\Delta_q v,\quad  R(u,v)=\sum_{q}\Delta_qu\tilde\Delta_{q}v  \quad\hbox{and}\quad \tilde\Delta_{q}=\sum_{i=-1}^1\Delta_{q+i}.
$$
For $(p,r)\in[1,+\infty]^2$ and $s\in\R$ we define  the homogeneous Besov \mbox{space $\Dot B_{p,r}^s$} 
as the set of  $u\in\mathcal{S}'(\R^d)$ such that $u = \sum_q \Delta_q u$ and
$$
\|u\|_{\dot B^s_{p,r}} =\Bigl\|\left( 2^{qs} \|\Delta
_qu\|_{L^p}\right)_{q\in \Z}\Bigr\|_{\ell^r(\Z)}<\infty.
$$
In the case $p=r=2$, the space $ \dot B^s_{2,2}$ turns out to be the classical homogeneous Sobolev space $\dot H^s$.  Finally, define $\tilde L^q_T \dot{B}^{s}_{p,r}$ is given by distributions $u$ such that
$$
\left\| u \right\|_{L^q_T \dot{B}^{s}_{p,r}} = \Bigl\|\left( 2^{qs} \|\Delta
_qu\|_{L^q_T L^p}\right)_{q\in \Z}\Bigr\|_{\ell^r(\Z)}<\infty.
$$

\section{Linear and Nonlinear estimates}
We will  make an extensive use of Bernstein's inequalities (see  \cite{Chemin95} for instance).
\begin{lemm}[Bernstein's lemma]
\label{lb}
 There exists a constant $C$ such that for any $q,k\in\N,$ $1\leq a\leq b$ and for  $f\in L^a(\R^d)$, 
\begin{eqnarray*}
\sup_{|\alpha|=k}\|\partial ^{\alpha}S_{q}f\|_{L^b}&\leq& C^k\,2^{q(k+d(\frac{1}{a}-\frac{1}{b}))}\|S_{q}f\|_{L^a},\\
\ C^{-k}2^
{qk}\|{\Delta}_{q}f\|_{L^a}&\leq&\sup_{|\alpha|=k}\|\partial ^{\alpha}{\Delta}_{q}f\|_{L^a}\leq C^k2^{qk}\|{\Delta}_{q}f\|_{L^a}.
\end{eqnarray*}
\end{lemm}

The parabolic regularity result we will need reads
\begin{lemm}[Parabolic regularization, see for example \cite{BCD}]
\label{hr}
Let $u$ be a smooth divergence free vector field solving
\begin{equation*}
   \label{u-f} 
   \left\{ \begin{array}{l} \partial_t u - \Delta
u + \nabla p
 = f \\
 u_{|t=0} = u^0, 
\end{array} \right. 
\end{equation*}
on some time interval $[0,T]$. Then, for  every $p\geq r \geq 1 $ and $s \in \R$ and $j\geq1$,
 \begin{equation*}
 \| u  \|_{C([0,T];\dot B^{s}_{q,j}  ) \cap \tilde L^p_T
\dot B^{s+\frac2p}_{q,j} }   \lesssim    \| u^0 \|_{\dot B^{s}_{q,j}} + \| f  \|_{\tilde L^r_T  \dot B^{s-2+\frac2r}_{q,j}}.
\end{equation*}
\end{lemm}

The following result is an $L^2_TL^\infty$ estimate which was originally proved 
in \cite{le}, \cite{CG} in dimension two.

\begin{lemm} [$L^2L^\infty$ estimate]
\label{hr refined}
Let $d=2,3$ and $u$ be a smooth divergence free vector field solving
\begin{equation*}
   \left\{ \begin{array}{l} \partial_t u - \Delta
u + \nabla p
 = f_1+f_2\\
 u_{|t=0} = u^0, 
\end{array} \right. 
\end{equation*}
on some time interval $[0,T]$. Assume that 
$f_1\in L^1_T \dot H^{\frac d2-1}$ and $f_2\in \tilde L^2_T \dot B^{\frac d2-2}_{2,1}$. Then,
 \begin{equation}
\label{L2Linfty}
 \|u\|_{L^2_TL^\infty }   \lesssim    \| u_0 \|_{\dot H^{\frac d2-1}} + 
\|f_1\|_{ L^1_T \dot H^{\frac d2-1}}+\|f_2\|_{\tilde L^2_T \dot B^{\frac d2-2}_{2,1}}.
\end{equation}
\end{lemm}

\begin{proof}
Thanks to Lemma \ref{hr}, and using the embeddings
$$
\tilde L^2 \dot{B}^{d/2}_{2,1} \hookrightarrow L^2 \dot{B}^{d/2}_{2,1} \hookrightarrow L^2 L^\infty,
$$
we can assume that $f_2=0$. Duhamel's formula gives
$$
u(t)=e^{t\Delta}u_0+\int_0^te^{(t-t')\Delta}{\mathcal P}f_1(t')\;dt',
$$
and thus 
\begin{equation}
\|u(t)\|_{L^2_TL^\infty} \leq \|e^{t\Delta} u_0\|_{L^2_TL^\infty}
+\int_0^T\|e^{t \Delta}{\mathcal P}f_1(t')\|_{L^2_T L^\infty}\;dt'.
\end{equation}
Using the embedding $\dot H^{\frac d2-1}\hookrightarrow \dot{\mathcal B}^{-1}_{\infty,2}$ 
and the characterization of Besov spaces of negative regularity (see for example \cite{BCD}),
$$
\|u\|_{\dot{\mathcal B}^{-1}_{\infty,2}}\sim 
\|\|e^{t\Delta}u\|_{L^\infty}\|_{L^2(0,\infty)},
$$
thus we obtain \eqref{L2Linfty} as desired. 
\end{proof}

Now we focus on Maxwell's equations.  The first result is an energy type estimate.

\begin{lemm}
\label{pr1}
Let $\alpha\geq0$, $G_1\in L^2_T\dot H^{\frac d2-1}_\alpha$, 
and $(E,B)$ be a smooth solution of 
\begin{eqnarray*}
\label{eq2B}
\partial_tE-\curl\; B+E&=&G\\
\partial_tB+\curl\; E&=&0\\
(E,B)_{|t=0} &=& (E_0,B_0),
\end{eqnarray*}
on some time interval $[0,T]$. 
Then, the following estimate holds (with constants independent of $T$)
\begin{equation}
\label{energy-bis}
\|E\|_{{\tilde L}^\infty_T \dot H^{\frac d2-1}_\alpha\cap 
L^2_T \dot H^{\frac d2-1}_\alpha}+
\|B\|_{{\tilde L}^\infty_T  \dot H^{\frac d2-1}_\alpha}
 \lesssim  \|(E_0,B_0)\|_{\dot H^{\frac d2-1}_\alpha}+
\|G\|_{L^2_T \dot H^{\frac d2-1}_\alpha}.
\end{equation}
Moreover, $B$ satisfies the following decay estimate
\begin{eqnarray}
 \label{B decay}
\|B\|_{L^2\dot H^{\frac d2,\frac d2-1}_\alpha}
 \lesssim  \|(E_0,B_0)\|_{\dot H^{\frac d2-1}_\alpha}+
\|G\|_{L^2_T \dot H^{\frac d2-1}_\alpha}.
\end{eqnarray}
\end{lemm}
We emphasize that for the existence and uniqueness part of Theorem~\ref{MAIN}, in dimension 3,
estimate \eqref{B decay}is irrelevant.
\begin{proof}  Only the estimate of 
$\|B\|_{L^2\dot H^{\frac d2,\frac d2-1}_\alpha}$
 requires a proof, which is given in the Appendix. All other estimates can be derived by a 
standard energy estimate; apply $\Delta_q$ to the system, derive an energy inequality, 
multiply both members of that inequality  by 
$2^{q\left( \frac{d}{2} - 1 \right)} {\sqrt{\max(1,q^\alpha)}}$ and take the $\ell^2(\mathbb Z)$-norm.
\end{proof}
The following is a series of nonlinear estimates needed for the contraction argument. 
\begin{prop} 
\label{lem-prod} 
For all smooth functions  $u$, $E$ and $B$ defined on some 
interval $[0,T]$, we have the following estimates, with constants independent of $T$: in space dimension 2,
\begin{eqnarray}
\label{est1 2D} 
\|\nabla(u\otimes v)\|_{L^1_TL^2(\R^2)}  \lesssim  
\|u \|_{L^2_T(L^\infty(\R^2)\cap \dot H^1(\R^2)) }
\|v \|_{L^2_T(L^\infty(\R^2)\cap \dot H^1(\R^2))}
\end{eqnarray}
\begin{eqnarray}
\label{est4 2D} 
\|E\times B \|_{\tilde L^2_T\dot {\mathcal B}^{-1}_{2,1}(\R^2) + L^1_T L^2(\mathbb{R}^2)} 
\lesssim   \|E\|_{{ L}^2_T L^2_{\rm log}(\R^2)} 
\|B\|_{{\tilde L}^\infty_T L^2_{\rm log}(\R^2) \cap L^2_T \dot H^{1,0}}
\end{eqnarray}
\begin{eqnarray}
\label{est3 uB d=2}
\|u\times B \|_{L^2_TL^2_{\rm log}(\R^2)}\lesssim
\|u\|_{L^2 (L^\infty(\R^2) \cap\dot H^1(\R^2) )}
\|B\|_{{\tilde L}^\infty_TL^2_{\rm log}(\R^2)},
\end{eqnarray}
and in space dimension 3,
\begin{eqnarray}
\label{est1 3D} 
\|\nabla(u\otimes v)\|_{L^1_T\dot H^{\frac12}(\R^3)}  \lesssim  
\|u \|_{L^2_T(L^\infty(\R^3)\cap \dot H^\frac32(\R^3))}
\|v \|_{L^2_T(L^\infty(\R^3)\cap \dot H^\frac32(\R^3))}
\end{eqnarray}
\begin{eqnarray}
\label{est4 3D} 
\|E\times B \|_{\tilde L^2_T\dot {\mathcal B}^{-\frac12}_{2,1}(\R^3)} 
\lesssim   \|E\|_{{L}^2_T \dot H^{\frac12}(\R^3)} 
\|B\|_{{\tilde L}^\infty_T \dot H^{\frac12}(\R^3)}
\end{eqnarray}
\begin{eqnarray}
\label{est3 uB d=3} 
\|u\times B\|_{{L}^2_T\dot H^{\frac12}(\R^3)}\lesssim
\|u\|_{L^2_T(L^\infty(\R^3)\cap \dot H^\frac32(\R^3))} 
\|B\|_{{\tilde L}^\infty_T\dot H^\frac12(\R^3)}.
\end{eqnarray}
\end{prop}

Estimates~\eqref{est1 2D} and~\eqref{est1 3D} enable us to control the advection term in the 
Navier-Stokes component of the system in dimension two and three, respectively. Estimates 
\eqref{est4 2D} and \eqref{est4 3D} are needed to control the Maxwell part 
in the Navier-Stokes component. To estimate the term $(u\times B)\times B$, we use 
\eqref{est4 2D}, \eqref{est3 uB d=2} in two space dimensions and 
\eqref{est4 3D}, \eqref{est3 uB d=3} in three space dimensions. 

\begin{rema} Ignoring the time variable, estimate \eqref{est4 3D} is a particular case of the product law
$$
H^{s_1}(\R^d){\boldsymbol \cdot} H^{s_2}(\R^d) \hookrightarrow \Dot B^{s_1+s_2-\frac{d}2}_{2,1}(\R^d)
$$
with  $s_1,s_2\in ]-\frac{d}2,\frac{d}2[$ and $s_1+s_2>0$. Indeed, it 
corresponds to the admissible choice $s_1=s_2=\frac12$. However, this product law becomes critical in two space dimensions.
Estimate \eqref{est4 2D} shows that an extra logarithmic loss is needed in this case.
\end{rema}
We give the proof of the above proposition in the Appendix.  



\section{Proof of Theorem \ref{MAIN}}

\subsection{Small data, global existence}
Let $\alpha=1$ if 
$d=2$ and $\alpha=0$ if $d=3$. Let ${\mathcal Z}$ be the set  of $\Gamma:=(u,E,B)^T$ such that
\begin{equation*}
\begin{split}
& u \in \mathcal{Z}^u := L^2(0,\infty,\dot H^\frac d2\cap L^\infty)\cap\tilde L^\infty(0,\infty,\dot H^{\frac d2-1}) \\
& E\in \mathcal{Z}^E := ({\tilde L}^\infty\cap L^2) (0,\infty,{\dot H}^{\frac d2-1}_\alpha) \\
& B\in \mathcal{Z}^B := {\tilde L}^\infty(0,\infty,\dot H^{\frac d2-1}_\alpha) \cap{L}^2(0,\infty,{H}^{\frac d2,\frac d2 - 1}_{\alpha}). \end{split}
\end{equation*}
Endow $\mathcal Z$, $\mathcal{Z}^u$, $\mathcal{Z}^E$ and $\mathcal{Z}^B$with the natural norms. Recall that we seek a solution to \eqref{NS} in the integral form
$$
\Gamma(t)=e^{t{\mathcal A}}\Gamma(0)+\int_0^te^{(t-t'){\mathcal A}}{\mathcal N}(\Gamma(t'))\;dt',
$$
with
$$
{\mathcal A}=\left(
\begin{array}{ccc}
\Delta&0&0\\
0&-I&\curl\\0&-\curl &0
\end{array}
\right).
$$
and  ${\mathcal N}(\Gamma)=\big({\mathcal P}(-\nabla(u\otimes u)+
E\times B+(u\times B)\times B), -u\times B,0\big)^T$. 
Let $B_{\delta}$ be the ball of the space ${\mathcal Z}_\infty$ 
centered at  zero and with radius $\delta>0$ to be chosen. 
Define the map $\Phi$ on that ball as follows
\begin{eqnarray}
\nonumber
 \Phi:B_{\delta}\subset{\mathcal Z}&\longrightarrow&{\mathcal Z}\\
 \label{contraction map}
\Gamma&\mapsto& \Phi(\Gamma):=\int_0^te^{(t-t'){\mathcal A}}{\mathcal N}
(e^{t'{\mathcal A}}\Gamma^0+\Gamma(t'))\;dt'.
\end{eqnarray}

\begin{claim}
If $\|\Gamma^0\|_{\dot H^{\frac d2-1}\times\dot H^{\frac d2-1}_\alpha
\times\dot H^{\frac d2-1}_\alpha}\leq \kappa \delta$, with $\delta>0$ and $\kappa>0$
sufficiently small, then the map $\Phi$ is a contraction on $B_{\delta}$. 
\end{claim}

The theorem follows immediately from the claim: Picard's theorem gives the existence of a fixed point
of $\Phi$, call it $\Gamma$. Then $e^{t{\mathcal A}}\Gamma^0 + \Gamma(t)$ is the desired solution.

\bigskip

\noindent
\textsc{Proof of the claim:} First, notice that $\Phi\left( - e^{t\mathcal{A}} \Gamma_0 \right) = 0$, while by Lemmas~\ref{hr},~\ref{hr refined} and~\ref{pr1},
\begin{equation}
\label{tree}
\left\| e^{t\mathcal{A}} \Gamma_0  \right\|_{\mathcal{Z}} \leq C \left\| \Gamma_0 \right\|_{\dot H^{\frac d2-1}\times\dot H^{\frac d2-1}_\alpha \times\dot H^{\frac d2-1}_\alpha} \leq C \kappa \delta \leq \frac{\delta}{2}
\end{equation}
for $\kappa$ small enough. On the other hand, we will prove below that, if $\Gamma_1$ and $\Gamma_2$ belong to $B_\delta$,
\begin{equation}
\label{arbre}
\left\| \Phi(\Gamma_1) - \Phi(\Gamma_2) \right\|_{\mathcal{Z}} \leq \frac{1}{2} \left\| \Gamma_1 - \Gamma_2 \right\|_{\mathcal{Z}}
\end{equation}
under the assumptions of the claim. 

The estimates~\eqref{tree} and~\eqref{arbre} easily yield the claim.

\bigskip

To prove~(\ref{arbre}), let $\Gamma_j:=(u_j,E_j,B_j)^T\in B_{\delta}$ for $j=1,2$. Write further
$$
e^{t{\mathcal A}}\Gamma^0+\Gamma_j(t)=(\bar u_j,\bar E_j,\bar B_j)
$$
and set $\Gamma:=\Gamma_1-\Gamma_2:=(u,E,B)^T$, and 
$\Phi(\Gamma_j):=\tilde \Gamma_j = ({\tilde u}_j,{\tilde E}_j,{\tilde B}_j)^T$ be given 
by \eqref{contraction map}. 
Let $\tilde\Gamma:=\tilde\Gamma_1-\tilde\Gamma_2
:=({\tilde u},{\tilde E},{\tilde B})^T$. 
We decompose  ${\tilde u}$ into ${\tilde u}=\tilde u^{NS}+ \tilde u^M$, with $\tilde u^{NS}$ accounting for the convection term
\begin{eqnarray}
\nonumber
\tilde u^{NS}:=-\int_0^te^{(t-t')\Delta}{\mathcal P}
\nabla(u\otimes \bar u_1 + \bar u_2
\otimes u)\;dt',
\end{eqnarray}
and $\tilde u^M$ for the Lorentz force
\begin{eqnarray*}
\tilde u^M:&=&\int_0^te^{(t-t')\Delta}{\mathcal P}\big(E\times\bar B_1+\bar E_2\times
 B\big)\;dt'\\
&+&\int_0^te^{(t-t')\Delta}{\mathcal P}\Big((u\times\bar B_1)\times\bar B_1
+[\bar u_2 \times B]\times\bar B_1+[\bar u_2 \times\bar B_2]\times B\Big)\;dt'.
\end{eqnarray*}
Moreover, the electro-magnetic field $(\tilde E,\tilde B)$ satisfies
\begin{eqnarray}
\label{max}
\partial_t \tilde E-\curl\; \tilde B+\tilde E&=&u\times\bar B_1+ \bar u_2\times B\\
\partial_t \tilde B+\curl\; \tilde E&=&0\
\end{eqnarray}
with $0$ data. First, by Lemma \ref{hr}, Lemma \ref{hr refined} and the embedding 
$L^1 \dot{H}^{\frac{d}{2}-1} \hookrightarrow \tilde L^1 \dot{H}^{\frac{d}{2}-1}$, we have 
\begin{eqnarray*}
\|\tilde u^{NS}\|_{\mathcal Z^u}&\lesssim& \|{\mathcal P}
\nabla(u\otimes \bar u_1 + \bar u_2
\otimes u)\|_{L^1{\dot H}^{\frac d2-1}}
\end{eqnarray*}
and
\begin{equation*}
\begin{split}
& \|\tilde u^M\|_{{\mathcal Z}^u} \lesssim \|{\mathcal P}\big(E\times\bar B_1+\bar E_2\times
 B\big)\|_{\tilde L^2\dot{B}^{\frac d2-2}_{2,1} + L^1 \dot{H}^{\frac d2 - 1}}\\
& \qquad \qquad +\|{\mathcal P}\Big((u\times\bar B_1)\times\bar B_1
+[ \bar u_2\times B]\times\bar B_1 + [\bar u_2 \times\bar B_2]\times B\Big)\|_{\tilde L^2\dot{B}^{\frac d2-2}_{2,1}+ L^1 \dot{H}^{\frac d2 - 1}}.
\end{split}
\end{equation*}
Second, applying estimates \eqref{est1 2D} and \eqref{est1 3D}, we obtain for the convection term
\begin{eqnarray}
\nonumber
 \| \tilde u^{NS} \|_{\mathcal Z^u}&\lesssim& 
\nonumber  \|u\|_{L^2 (L^\infty \cap \dot H^{\frac d2})} \sum_{j=1,2} \|\bar u_j\|_{ L^2 (L^\infty \cap \dot H^{\frac d2})}\\
&\lesssim&\|\Gamma\|_{{\mathcal Z}}\sum_{j=1,2}\left( \|\Gamma_j\|_{{\mathcal Z}}+
\|e^{t\mathcal A}\Gamma^0\|_{\mathcal Z} \right),
\label{est1 v}
\end{eqnarray}
whereas the Lorentz force term can be estimated by~(\ref{est4 2D}),~(\ref{est3 uB d=2}),~(\ref{est4 3D}), \and~(\ref{est3 uB d=3})
\begin{equation}
\label{estM}
\begin{split}
& \|\tilde u^M \|_{{\mathcal Z^u}} \lesssim 
\|E\|_{L^2\dot H^{\frac d2-1}_\alpha} \|\bar B_1\|_{\tilde L^\infty\dot H^{\frac d2-1}_\alpha \cap L^2 \dot{H}^{\frac d2,\frac d2 -1} }
+ \|\bar E_2\|_{L^2\dot H^{\frac d2-1}_\alpha} \| B \|_{\tilde L^\infty\dot H^{\frac d2-1}_\alpha \cap L^2 \dot{H}^{\frac d2,\frac d2 -1}} \\
& \quad + \| u \times \bar B_1\|_{L^2\dot H^{\frac d2-1}_\alpha} \|\bar B_1\|_{\tilde L^\infty\dot H^{\frac d2-1}_\alpha \cap L^2 \dot{H}^{\frac d2,\frac d2 -1}}
+ \| \bar u_2 \times B\|_{L^2\dot H^{\frac d2-1}_\alpha} \|\bar B_1\|_{\tilde L^\infty\dot H^{\frac d2-1}_\alpha \cap L^2 \dot{H}^{\frac d2,\frac d2 -1}} \\
& \quad
+ \| \bar u_2 \times \bar B_2 \|_{L^2\dot H^{\frac d2-1}_\alpha} \|\bar B\|_{\tilde L^\infty\dot H^{\frac d2-1}_\alpha} \\
& \lesssim \|E\|_{\mathcal{Z}^E} \|\bar B_1\|_{\mathcal{Z}^B} + \|\bar E_2\|_{\mathcal{Z}^B} \| B \|_{\mathcal{Z}^B} + \| u \|_{\mathcal{Z}^u} \|\bar B_1\|_{\mathcal{Z}^B} \|\bar B_1\|_{\mathcal{Z}^B} + \| \bar u_2 \|_{\mathcal{Z}^u} \| B\|_{\mathcal{Z}^B}  \|\bar B_1\|_{\mathcal{Z}^B} \\
& \quad + \| \bar u_2 \|_{\mathcal{Z}^u} \|\bar B_2\|_{\mathcal{Z}^B}  \| B\|_{\mathcal{Z}^B} \\
& \lesssim \left\| \Gamma \right\|_{\mathcal{Z}} \sum_{j=1,2} \left[ 
\|e^{t\mathcal A} \Gamma^0\|_{\mathcal{Z}} + \|e^{t\mathcal A} \Gamma^0\|_{\mathcal{Z}}^2 + \left\| \Gamma_j \right\|_{\mathcal{Z}} 
+  \left\| \Gamma_j \right\|_{\mathcal{Z}}^2 \right].
\end{split}
\end{equation}
It remains to estimate the electro-magnetic field components of $\Gamma$. Applying the energy and the decay estimates \eqref{energy-bis} to the 
system \eqref{max}, we get
\begin{equation}
\label{est2 Gamma}
\|\tilde E\|_{\mathcal{Z}^E} + \|\tilde B\|_{\mathcal{Z}^B} 
\lesssim\|\Gamma\|_{\mathcal Z}
\sum_{j=1}^2(\|e^{t\mathcal A} \Gamma^0\|_{\mathcal{Z}}+\|\Gamma_j\|_{\mathcal Z}).
\end{equation}
Gathering the estimates \eqref{est1 v}, \eqref{estM} and \eqref{est2 Gamma} gives
\begin{equation}
\| \tilde \Gamma \|_{\mathcal{Z}} \lesssim \|\Gamma\|_{\mathcal{Z}} \left( \|e^{t\mathcal A} \Gamma^0\|_{\mathcal{Z}} + \|e^{t\mathcal A} \Gamma^0\|_{\mathcal{Z}}^2 + \|\Gamma_j\|_{\mathcal Z} + \|\Gamma_j\|_{\mathcal Z}^2 \right).
\end{equation}
Choosing $\delta$ small enough gives~(\ref{arbre}).

\subsection{The local existence}
Decompose the initial data $(u^0,E^0,B^0)=(u^0_r,E^0_r,B^0_r)+(u^0_s,E^0_s,B^0_s)$
 where $(u^0_r,E^0_r,B^0_r)$ is regular (say in $H^2$) and $(u^0_s,E^0_s,B^0_s)$
is small in $\dot H^{\frac d2-1} \times \dot H^{\frac d2-1}_\alpha \times \dot H^{\frac d2-1}_\alpha$ (this can be done by a Fourier cut-off). 
We look for a solution $\Gamma$ of \eqref{NS} of the form 
$(u,E,B) = (u_s,E_s,B_s) + (u_r,E_r,B_r)$ with
\begin{equation}
\label{NS appendix}
\ \left\{
\begin{array}{rclll}
\frac{\partial u_r}{\partial t} +   u_r\cdot\nabla  u_r - 
    \Delta u_r +\nabla p_r & =&     j_r {{\times}}    B_r \\
\partial_t E_r  - \curl\, B_r&=&-j_r \\
\partial_tB_r  +  \curl\, E_r
  & = & 0 & \\
 \div u_r=\div B_r &=&0
\\
 ( E_r + u_r {{\times}}    B_r) &=&j_r
\end{array}
\right.
\end{equation}
subject to the initial data 
\begin{equation*}
 \label{ini appendix}
{u_r}_{|t=0} = u^0_r, \quad {B_r}_{|t=0} = B^0_r, \quad {E_r}_{|t=0}= E^0_r.
\end{equation*}
Arguing as in \cite{IY}, we know that \eqref{NS appendix} 
has a unique regular solution. Now we solve for $(u_s,E_s,B_s)$. We have 
\begin{equation}
\label{NS perturbed}
\ \left\{
\begin{array}{rclll}
\frac{\partial u_s}{\partial t} + u_s \cdot\nabla u_s - 
\Delta u_s + u_s \cdot\nabla u_r+ u_r\cdot\nabla u_s
+\nabla p_s & =&  j \times B - j_r\times B_r \\
\partial_t E_s  - \curl\, B_s &=& j - j_r \\
\partial_t B_s  +  \curl\, E_s & = & 0 & \\ 
\div u_s =\div B_s &=&0
\end{array}
\right.
\end{equation}
subject to the initial data 
\begin{equation*}
{u_s}_{|t=0} = u_s^0, \quad { B_s}_{|t=0} = B_s^0, 
\quad {E}_{|t=0}= E^0_s.
\end{equation*}
Proceeding similarly as for the small data result, set
$$
\Phi(\Gamma) := \int_0^t e^{(t-t')\mathcal{A}} 
\left( 
\begin{array}{l} 
\mathcal{P} \left[ - u_s \cdot\nabla u_s - u_s \cdot\nabla u_r - u_r\cdot\nabla u_s + j \times B - j_r\times B_r \right] \\
j - j_r
\end{array}
\right) dt'
$$
where $\Gamma$ is defined by
$$
(u_s,B_s,E_s) = e^{t\mathcal A} (u_s^0,B_s^0,E_s^0) + \Gamma.
$$

Applying the same proof as for the small data existence, 
we can show that the map $\Phi$ is a contraction if we choose a time 
of existence $T$ sufficiently small. The main difference is that new linear terms
(in $\Gamma$) appear in $\Phi$. These linear terms need to be small (as linear maps)
for $\Phi$ to be a contraction; this can be achieved by using the smoothness of
$B$ and by choosing $T$ small enough.

For instance:
\begin{equation}
\begin{split}
\left\| \int_0^t e^{(t-t')\Delta}{\mathcal P} E_s \times B_r \,dt' \right\|_{\mathcal{Z}^u_T} & \lesssim \left\| E_s \times B_r \right\|_{L^1_T \dot H^{\frac{d}{2}-1}}
\lesssim \left\| E_s \right\|_{L^\infty_T \dot H^{\frac{d}{2}-1}} \left\| B_r \right\|_{L^1_T H^2} \\
& \lesssim \left\| E_s \right\|_{\mathcal{Z}^E_T} \left\| B_r \right\|_{L^1_T H^2}.
\end{split}
\end{equation}
The key point is of course that $\left\| B_r \right\|_{L^1_T H^2}$ can be made arbitrarily small by choosing $T$ small enough.


\section{Appendix: Proof of the decay for $B$ and Proposition~\ref{lem-prod} }
Recall the main part of Lemma \ref{hr refined}. We emphasize on the 
fact that this property is an extra information about a weak decay 
that the magnteic field satisfies and it has no impact on the well 
posedness result.

\begin{lemm}
Let $\alpha\geq0$ and $G\in L^2_T\dot H^{\frac d2-1}_\alpha$, 
and $(E,B)$ be a solution of 
\begin{eqnarray*}
\partial_tE-\curl\; B+E&=&G,\\
\partial_tB+\curl\; E&=&0\
\end{eqnarray*}
on some time interval $[0,T]$. 
Then, the following estimate hold with constants not depending upon time
\begin{equation}
\|B\|_{L^2\dot H^{\frac d2,\frac d2-1}_\alpha}
 \lesssim  \|(E_0,B_0)\|_{\dot H^{\frac d2-1}_\alpha}+
\|G\|_{L^2_T \dot H^{\frac d2-1}_\alpha}.
\end{equation}
\end{lemm}

\begin{proof}
Because of the divergence free property of $B$, we have

\begin{eqnarray}
\label{equ B}
\partial_{tt}B-\Delta B+\partial_tB=\hbox{curl}\,G,\qquad (B,\partial_tB)_{|t=0}=(B^0,B^1).
\end{eqnarray}

Thus, the magnetic field $B$ satisfies an inhomogeneous damped wave equation \eqref{equ B}.
In the sequel we denote by 
$\mathcal L_1(t)$ and $\mathcal L_2(t)$ the propagators associated to the Fourier multiplier functions
$$
\Phi_1(t,\xi)=e^{-t/2} \cosh({\sqrt{1/4-|\xi|^2}}t),\qquad \Phi_2(t,\xi)=e^{-t/2}\frac{
  \sinh({\sqrt{1/4-|\xi|^2}}t)}{\sqrt{1/4-|\xi|^2}}.
$$
A direct computation gives the following Duhamel type formula
\begin{equation*}
{B}(t,x)=\mathcal L_1(t){B}^0(x)+ \mathcal L_2(t)(B^0/2+{B}^1)(x)+ \int_0^t\mathcal L_2(t-s){ \hbox{curl}\, (G)}(s,x)ds,
\end{equation*}
with $B^1=\partial_tB(t=0)=-\hbox{curl}\;( E(t=0))=-\hbox{curl}\;( E^0)$. As this was observed in \cite{IK}, there exists $0<c<1$ such that we have the following bounds\\
$\bullet$ For $|\xi|\geq 2$ 
\begin{eqnarray*}
\label{high1}
|\Phi_1(t,\xi)|& \lesssim & e^{-{ct}}
\\
\label{high2}
|\Phi_2(t,\xi)|& \lesssim &\frac{e^{-ct}}{|\xi|}.
\end{eqnarray*}
$\bullet$ For ${1/4}\leq|\xi|<2$
\begin{eqnarray*}
\label{medium}
|\Phi_1(t,\xi)|+ |\Phi_2(t,\xi)| \lesssim  e^{-{ct}}
\end{eqnarray*}
$\bullet$ For $2^{q-1}\leq|\xi|\leq2^{q+1}$, with $q\leq{-3}$
\begin{eqnarray*}
|\Phi_1(t,\xi)|&\leq& \Phi_q^1(t):=e^{-\frac{t}2} \cosh(t\sqrt{1/4-2^{2(q{-1})}})
\\
|\Phi_2(t,\xi)|&\leq& \Phi_q^2(t):=e^{-\frac{t}2}\frac{  \sinh(t\sqrt{1-2^{2(q{-1})}})}{\sqrt{1/4-2^{2(q-1)}}}.
\end{eqnarray*}
On the one hand, for  $q\geq -1$, one has
\begin{eqnarray}
 \label{hig freq B}
\|\Delta_qB(t)\|_{L^2} &\lesssim&  e^{-{ct}}\|\Delta_q B^0\|_{L^2}+e^{-{ct}}2^{-q}\big(\|\Delta_q B^0\|_{L^2}
+\|\Delta_q B^1\|_{L^2}\big)\\&+& \int^{t}_0e^{{-c{(t-s)}}}\|\Delta_qG\|_{L^2}ds.
\nonumber
\end{eqnarray}
Multiplying both sides of \eqref{hig freq B} by $2^{q(\frac d2-1)}$, applying Young's inequality (in time) and summing in $q$ yields

\begin{eqnarray}
\|(I-S_0)B\|_{ L^2_T {\dot H}^{\frac d2-1}_\alpha}&\lesssim&  
\|(I-S_0)B^0\|_{\dot H^{\frac d2-1}_\alpha}+
\|(I-S_0)B^1\|_{H^{\frac d2-2}_\alpha}\label{step1}\\
\nonumber&+&
\|(I-S_0)G\|_{L^2_T \dot H^{\frac d2-1}_\alpha}
\end{eqnarray}
On the other hand, for $q\leq0$, one has 
\begin{eqnarray*}
\|\Delta_q B(t)\|_{L^2}&\leq&  \Phi_q^1(t) \|\Delta_qB^0\|_{L^2}+ \Phi_q^2(t) {\Big( \|\Delta_qB^1\|_{L^2}+\|\Delta_qB^0\|_{L^2}\Big)}
\\
&+& 2^q\int^t_0 \Phi_q^2(t-s) \|\Delta_qG(s)\|_{L^2}ds.
\end{eqnarray*}
Taking the $L^2_T$ norm in time and applying Young's inequality we get 
\begin{eqnarray*}
\|\Delta_q B\|_{L^2 L^2}& \lesssim &    \| \Phi_q^1\|_{L^2(\R^+)} \|\Delta_qB^0\|_{L^2}
\\
&{+}&\,    \|\Phi_q^2\|_{L^2(\R^+)}{\Big( \|\Delta_qB^1\|_{L^2}+\|\Delta_qB^0\|_{L^2}\Big)}
\\
&{+}&\,    2^q\|\Phi_q^2\|_{L^1(\R^+)} \|\Delta_qG\|_{L^2_TL^2}.
\end{eqnarray*}
But since for every $q\leq 0$ and $r\in [1,+\infty]$, $\Phi_q^i$ satisfies 
$$
\|\Phi_q^i\|_{L^r(\R^+)} \lesssim  2^{-\frac2rq},\qquad i=1,2.
$$
Multiplying both sides by $2^{q\frac d2} q^{\alpha/2}$ and taking the $\mathcal{\ell}^2$ norm gives
\begin{eqnarray}
\label{step2}
\|S_0B\|_{L^2_T \dot H^{\frac d2}} &\lesssim&  
\|S_0B^0\|_{\dot H^{\frac d2-1}_\alpha}+\|S_0B^1\|_{\dot H^{\frac d2-2}_\alpha} + \|S_0G\|_{L^2_T \dot H^{\frac d2-1}_\alpha}.
\end{eqnarray}
Putting together \eqref{step1} and 
\eqref{step2} gives \eqref{energy-bis} as desired.
\end{proof}

{\sc Proof of Proposition 3.5}
The proof is based on the paraproduct decomposition. We choose to prove in 
details only estimates \eqref{est4 2D} and \eqref{est3 uB d=2}. The other estimates are easier or classical
and left to the reader.

\bigskip

\noindent \underline{Proof of~\eqref{est4 2D}} We decompose $EB$ into
$$
EB=T_EB+T_BE+S_2R( E,B)+(I-S_2)R( E,B), 
$$
and will show the following estimates: 
\begin{eqnarray}
\label{est3a 2D}
\|T_{E}B+T_BE\|_{\tilde L^2_T\dot{B}^{-1}_{2,1}(\R^2)}&\lesssim&
\|E\|_{L^2_TL^2(\R^2)}\|B\|_{{\tilde L}^\infty_T L^2(\R^2)}
\end{eqnarray}
\begin{eqnarray}
\label{est3b 2D}
\|S_2R(E,B)\|_{L^1_T L^2}  &\lesssim&  
\|E \|_{L^2_T L^2(\R^2)}  \|B\|_{L^2_T\dot H^{1,0}(\R^2)}
\end{eqnarray}
\begin{eqnarray}
\label{est3c 2D}
\|(I-S_2)R(E,B) \|_{{\tilde L}^2_T\dot{\mathcal B}^{-1}_{2,1}(\R^2)} \lesssim
  \|E \|_{L^2_T L^2_{\rm log}(\R^2)} \|B\|_{\tilde L^\infty_T L^2_{\rm log}} 
\end{eqnarray}
First, we prove \eqref{est3a 2D}.  Since the term $T_B E$ can be treated in a very similar way, we focus on $T_E B$. First,
$$
\Delta_q(T_EB)=\sum_{|{\tilde q}-q|\leq 1}\Delta_q(\Delta_{\tilde q}B S_{\tilde q}E).
$$
Since $\Delta_q$ is uniformly bounded on $L^2$, we have
$$
\sum_{q}2^{-q}\|\Delta_q(T_EB)\|_{ L^2_TL^2} \lesssim  \sum_{q}2^{-q}\sum_{|{\tilde q}-q|\leq 1} \|\Delta_{\tilde q}B S_{\tilde q}E\|_{ L^2_TL^2}.
$$
We are going to deal with the term ${\tilde q}=q$ only (the two other terms ${\tilde q}=q\pm 1$ can be estimated similarly). Applying successively H\"older's inequality (in the variables $t$ and $x$), Bernstein's lemma, Young's inequality (in the variable $q$), and H\"older's inequality (in the variable $q$) gives
\begin{equation*}
\begin{split}
\sum_{q}2^{-q} \|\Delta_{q}B S_{q}E\|_{ L^2_TL^2} & \leq 
\sum_{q}2^{-q} \|\Delta_{q}B \|_{L^\infty_T L^2} \| S_q E\|_{ L^2_TL^\infty} \\
& \leq \sum_{q}2^{-q}  \sum_{j \leq q} \|\Delta_{q}B \|_{L^\infty_T L^2} \| \Delta_{j}E\|_{ L^2_TL^\infty} \\
& \leq \sum_{q} \sum_{j \leq q} 2^{j-q} \|\Delta_{q}B  \|_{L^\infty_T L^2} \| \Delta_{j}E\|_{ L^2_TL^2} \\
& \leq \left( \sum_q \| \Delta_{q}B  \|_{L^\infty_T L^2}^2  \right)^{1/2} \left( \sum_j \| \Delta_{j}E\|_{ L^2_TL^2}^2 \right)^{1/2}.
\end{split}
\end{equation*}
Next we prove \eqref{est3b 2D}. Applying Bernstein's Lemma \ref{lb} and Cauchy-Schwarz (in $j$) gives
\begin{eqnarray*}
\|S_{2}R(B,E)\|_{\tilde L^1_T L^2 }& \lesssim &
\sum_{q\leq  0} \|\Delta_qR( B,E)\|_{L^1_T L^2} \\
& \lesssim & \sum_{q\leq  0}2^{q} \|\Delta_qR( B,E)\|_{L^1_T L^1} \\
 & \lesssim & \sum_{q\leq 0}2^{q} \sum_{j\geq q-2}\|\D_jB\|_{{ L}^2_T L^2}\|\D_jE\|_{{ L}^2_T L^2} \\
& \lesssim & \sum_{j} \sum_{q \leq \inf(0,j+2)} 2^q \|\D_jB\|_{{ L}^2_T L^2}\|\D_jE\|_{{ L}^2_T L^2} \\
& \lesssim & \sum_{j \leq 0} 2^j  \|\D_jB\|_{{ L}^2_T L^2}\|\D_jE\|_{{ L}^2_T L^2} + 
\sum_{j \geq 0} \|\D_jB\|_{{ L}^2_T L^2}\|\D_jE\|_{{ L}^2_T L^2} \\
& \lesssim  & \|E\|_{L^2_TL^2}\|B\|_{{ L}^2_T\dot H^{1,0}}.
\end{eqnarray*}
To estimate~(\ref{est3c 2D}), H\"older's inequality (in $t,x$) and Cauchy-Schwarz (in $j$) gives
\begin{eqnarray*}
\|(I-S_2)R(E,B) \|_{{\tilde L}^2_T\dot{\mathcal B}^{-1}_{2,1}}   &\lesssim& \sum_{q\geq0}\sum_{j\geq q-2} \|\D_jE\|_{ L^2_TL^2} \| \D_jB\|_{L^\infty_TL^2}\\
&\lesssim& \sum_{j \geq -2}\sum_{0\leq q\leq j+2} \|\D_jE\|_{ L^2_TL^2} \| \D_jB\|_{L^\infty_TL^2}\\
&\lesssim& \sum_{j \geq -2} \max(j,1) \|\D_jE\|_{ L^2_TL^2} \| \D_jB\|_{L^\infty_TL^2}\\
&\lesssim& \|E\|_{L^2_T L^2_{\rm log}} \|B\|_{L^\infty_T L^2_{\rm log}}.
\end{eqnarray*}

\bigskip

\noindent \underline{Proof of~\eqref{est3 uB d=2}} As for the proof of~(\ref{est4 2D}), we split $uB$ following the paraproduct decomposition:
$$
uB = T_B u + T_u B + R(u,B).
$$
We shall only estimate here $T_B u$, the estimate of $R(u,B)$ being similar, and that of $T_u B$ easier. By H\"older's inequality,
\begin{equation*}
\begin{split}
\left\| T_B u \right\|_{L^2_T L^2_{\rm log}}^2 & = \sum_q \max(1,q) \left\| S_q B \Delta_q u \right\|_{L^2_T L^2}^2 \\
& \lesssim \sum_q \max(1,q)  \left\| \Delta_q u \right\|_{L^2_T L^2}^2 \left\| S_q B \right\|_{L^\infty_T L^\infty}^2.
\end{split}
\end{equation*}
Now observe that Bernstein's lemma and Cauchy-Schwarz' inequality (in $j$) give
\begin{equation*}
\begin{split}
\left\| S_q B \right\|_{L^\infty_T L^\infty} & \lesssim \sum_{j < q} 2^j \left\| \Delta_j B \right\|_{L^\infty_t L^2} \\
& \lesssim \left( \sum_{j<q} \frac{2^{2j}}{\max(1,q)} \right)^{1/2} \|B\|_{\tilde L^\infty L^2_{\rm log}} \\
& \lesssim \frac{2^q}{\sqrt{\max(1,q)}}.
\end{split}
\end{equation*}
Coming back to the bound for $\left\| T_B u \right\|_{L^2_T L^2_{\rm log}} $, this gives
\begin{equation*}
\begin{split}
\left\| T_B u \right\|_{L^2_T L^2_{\rm log}}^2 \lesssim \sum 2^{2q} \|\Delta_q u \|_{L^2_T L^2}^2 \|B\|_{\tilde L^\infty_T L^2_{\rm log}}^2.
\end{split}
\end{equation*}

\end{document}